\documentclass[a4paper,12pt,centertags]{amsart}
\usepackage[latin1]{inputenc}
\usepackage[english]{babel}
\usepackage{cite}
\usepackage{verbatim}
\usepackage{color}
\usepackage{txfonts}
\usepackage{pgf,tikz}
\usetikzlibrary{arrows}
\usepackage{graphicx,picture}
\usepackage{amsmath,amsthm}
\usepackage{comment}
\usepackage{hyperref}
\usepackage[top=3.2cm, bottom=3.2cm, left=2.8cm,
right=2.8cm]{geometry}

\newtheorem{theo}{Theorem}[section]

\newtheorem{prop}{Proposition}[section]

\newtheorem{cor}{Corollary}[section]

\theoremstyle{definition}

\newtheorem{rem}{Remark}[section]

\numberwithin{equation}{section}

\newcommand{\norm}[1]{\left\Vert#1\right\Vert}

\newcommand{\R}{\mathbb R}
\newcommand{\W}{\mathcal W}
\newcommand{\de}{\partial}
\newcommand{\eps}{\varepsilon}

\DeclareMathOperator{\divergenza}{div}

\DeclareMathOperator{\sign}{sign}

\begin{document}
\title[Anisotropic elliptic equations]{Anisotropic elliptic equations
  with general growth in the gradient and Hardy-type potentials}
\author[F. Della Pietra, N. Gavitone]{Francesco Della Pietra and
  Nunzia Gavitone}
\address{Francesco Della Pietra \\
Universit\`a degli studi del Molise \\
Dipartimento di Bioscienze e Territorio \\
Via Duca degli Abruzzi \\
86039 Termoli (CB), Italia.
}
\email{francesco.dellapietra@unimol.it}

\address{Nunzia Gavitone \\
Universit\`a degli studi di Napoli ``Federico II''\\
Dipartimento di Matematica e Applicazioni ``R. Caccioppoli''\\
80126 Napoli, Italia.
}
\email{nunzia.gavitone@unina.it}
\keywords{Nonlinear elliptic boundary value problems, Hardy
  inequalities, a priori estimates, convex symmetrization}
\subjclass{35J65,35B45}
\date{\today}
\maketitle
\begin{abstract}
In this paper we give existence and regularity results for the
solutions of problems whose prototype is
\begin{equation*}
  \left\{
    \begin{array}{ll}
      -\mathcal Q v= \beta(|v|) H(Dv) ^ q +
      \dfrac{\lambda}{H^o(x)^p} |v|^{p-2} v + f(x) &
      \text{in }\Omega, \\ [.2cm]
      v = 0 & \text{on }\de\Omega,
    \end{array}
  \right.
\end{equation*}
with $\Omega$ bounded domain of $\R^N$, $N\ge 2$, $0<p-1<q\le p<N$,
$\beta$ is a nonnegative continuous function and $\lambda \ge
0$. Moreover, $H$ is a general norm of $\R^N$, $H^o$ is its polar and
$\mathcal Q v:= \sum_{i=1}^{N} \frac{\de}{\de x_i}
\big(H(Dv)^{p-1}H_{\xi_i}(Dv)\big)$. 
\end{abstract}

\section{Introduction}

In the present paper we study existence and regularity results for
Dirichlet problems which involve a class of nonlinear elliptic
operators in divergence form, under the influence of lower-order
terms. Given a function $H\colon \R^N\rightarrow
[0,+\infty[$, $N\ge 2$, convex, $1$-homogeneous and in
$C^1(\R^N\setminus\{ 0 \})$, we deal with operators whose prototype is
the following: 
\begin{equation}
  \label{intro:p=2}
  \mathcal Q v:= \sum_{i=1}^{N} \frac{\de}{\de x_i}
  \big(H(Dv)^{p-1}H_{\xi_i}(Dv)\big),
\end{equation}
with $1<p<N$. In general, $\mathcal Q $ is highly nonlinear, and
extends some well-known classes of operators. In particular, for
$H(\xi)= ( \sum_k |\xi_k|^r )^\frac{1}{r}$, $r>1$, $\mathcal Q$ becomes
\begin{equation*}
  \mathcal Q v= \sum_{i=1}^{N} \frac{\de}{\de x_i} \left(
    \left( \sum_{k=1}^{N} \left| \frac{\de v}{\de x_k} \right|^r
    \right)^{(p-r)/r} \left| \frac{\de v}{\de x_i}
    \right|^{r-2}\frac{\de v}{ \de x_i}\right).
\end{equation*}
Note that for $r=2$, it coincides with the usual $p$-Laplace
operator, while for $r=p$ it is the so-called pseudo-$p$-Laplace
operator.

This kind of operator has been studied in several papers
(see for instance \cite{aflt}, \cite{ciasal}, \cite{dpg2},
\cite{dpg3}, \cite{ferkaw} for $p=2$, and \cite{bfk}, \cite{bkj06},
\cite{kn08} for $1<p<\infty$).

The aim of this paper is to study a class of equations whose prototype
involves in its principal part the operator \eqref{intro:p=2}, and 
a Hardy-type potential. Moreover, we are also interested in the
influence of a lower-order term depending on the gradient. The
problems we deal with are modeled on the following: 
\begin{equation}\label{eq:radH}
  \left\{
    \begin{array}{ll}
      -\mathcal Q v= \beta(|v|) H(Dv) ^ q +
      \dfrac{\lambda}{H^o(x)^p} |v|^{p-2} v + f(x) &
      \text{in }\Omega, \\ [.2cm]
      v = 0 & \text{on }\de\Omega,
    \end{array}
  \right.
\end{equation}
with $\Omega$ bounded domain of $\R^N$, with $0\in\Omega$, $N\ge 2$, $1<p<N$, 
$p-1<q\le p$, $\beta$ is a nonnegative continuous function, $\lambda
\ge 0$ and $f$ a measurable function on whose summability we will make
different assumptions. Moreover, we denote with $H^o$ the
polar function of $H$ (see Section 2 for the precise definition). 

When $H(\xi)=|\xi|$, the general problem \eqref{eq:radH} reduces to 
\begin{equation}\label{eq:plap}
  \left\{
    \begin{array}{ll}
      -\Delta_p v= \beta(|v|) |Dv| ^ q +
      \dfrac{\lambda}{|x|^p} |v|^{p-2} v + f(x) &
      \text{in }\Omega, \\ [.2cm]
      v= 0 & \text{on }\de\Omega.
    \end{array}
  \right.
\end{equation}
Equations like \eqref{eq:plap} have been widely studied in literature
either  in the case $\lambda=0$ or when
$\beta=0$.

In the case $\lambda=0$, it is well-known that for a
general continuous function $\beta$, a smallness assumption on
some norm of $f$ is needed in order to have existence results (see,
for example, \cite{fm98,fm00,fermes,gmp06,gmp12} for $\beta\equiv 1$, or
\cite{abddaper,dp,hmv99,mps,tr03} in the general case). 
Moreover, under some appropriate hypotheses on the function $\beta$,
it is possible to remove the smallness condition of $f$ (see
\cite{bst01,sp06}).

In the case $\beta=0$, the existence of a solution of \eqref{eq:plap}
can be proved under the assumption of $\lambda\le \Lambda_N$ (see
\cite{gaper}), where $\Lambda_N$ denotes the best constant in the
classical Hardy inequality. Moreover, if $p=2$ in
\cite{bop} some regularity results are proved. Surprisingly, the
regularity of the solutions also depend on the size of $\lambda$. 

As matter of fact, the influence of both terms in the right-hand side
of \eqref{eq:plap} has been studied in \cite{app07,abdper07} in
the case $\beta$ is a positive constant. In such papers, some
existence and nonexistence results are proved. In particular, it is
shown that when $p=q$ there is no positive solution, even in a very weak
sense, when $f>0$ and $\lambda >0$.

Recently, in \cite{lmp} the authors study problems whose model is
\eqref{eq:plap} with $q=p$ and $\beta$ nonconstant, giving some
existence and regularity results. More precisely, they prove that
under a structural assumption on $\beta$, which involves its behavior
at infinity, if $f\in L^{m}(\Omega)$, $m>1$ there exists a solution of
\eqref{eq:plap} whose regularity depends on $m$ and on the size of
$\lambda$. 

As regards the general problem \eqref{eq:radH}, in \cite{dpg3} we
investigated the particular case with $\beta\equiv 0$ and $p=2$, namely  
\begin{equation}
  \label{intro:mod}
  \left\{
    \begin{array}{ll}
      -\mathcal Q v = \dfrac{\lambda}{H^o(x)^2}v + f(x) &\text{in }
      \Omega,\\[.1cm]
      v=0 &\text {on }\de \Omega.
    \end{array}
  \right.
\end{equation}
Here $\Omega$ is a bounded open set of $\R^N$,
$N\ge 3$, containing the origin, and $\lambda$ is a nonnegative
constant. We studied the existence and the 
regularity of the solutions of \eqref{intro:mod} with respect to the
summability of $f$, chosed in suitable Lorentz spaces, and the size of
$\lambda$.

Our purpose is to study problem \eqref{eq:radH} for a
general $\beta\ge 0$ and $p-1<q\le p$. In particular, the
novelties of the paper relies in two main topics. First, using
symmetrization techniques we are able to fully analyse the case $q<
p$ that, up to our knowledge, also in the Euclidean case
has been studied only in particular cases. 
Second, taking into account the structure of the operator, we use a
suitable symmetrization argument, involving the so-called convex
symmetrization (see \cite{aflt}, and Section 2 for the definition),
which allows to obtain optimal results (see Remark \ref{opth}).

To study problem \eqref{eq:radH}, we investigate the existence and
regularity issues by choosing $f$ in appropriate Lorentz spaces. Under
suitable assumptions on $\beta$, we find a critical value of
$\lambda$, which depends on $\beta$ and on the 
summability of $f$, such that a solution of \eqref{eq:radH}
exists. Moreover, we prove that the obtained solution  and its
gradient belong to suitable Lorentz spaces (see Section 3, Theorems
\ref{theo:deb}, \ref{theo:entr}).  

As usual, a key role is played by uniform estimates of the solutions
of appropriate approximating problems (see Section 4), obtained by
means of the quoted convex symmetrization.

For ease of reading, we state the main results in Section 3, adding
some comments and remarks. Their proofs are contained in sections 4
and 5.


\section{Notation and preliminaries}
Let $N\ge 2$, and $H:\R^N\rightarrow [0,+\infty[$ be a $C^1(\mathbb
R^N\setminus\{0\})$ function such that
\begin{equation}\label{eq:omo}
  H(t\xi)= |t| H(\xi), \quad \forall \xi \in \R^N,\; \forall t \in
  \R.
\end{equation}
Moreover, suppose that there exist two positive constants $c_1
\le c_2$ such that
\begin{equation}\label{eq:lin}
  c_1|\xi| \le H(\xi) \le c_2|\xi|,\quad \forall \xi\in \R^N.
\end{equation}

The polar function $H^o\colon \R^N\rightarrow [0,+\infty[$
of $H$ is defined as
\[
H^o(v)\colon=\sup_{\xi \ne 0} \frac{\xi\cdot v}{H(\xi)}.
\]
It is easy to verify that also $H^o$ is a convex function
which satisfies properties \eqref{eq:omo} and
\eqref{eq:lin}. Furthermore,
\[
H(v)=\sup_{\xi \ne 0} \frac{ \xi \cdot v}{H^o(\xi)}.
\]
The set
\[
\mathcal W = \{  \xi \in \R^N \colon H^o(\xi)< 1\}.
\]
is the so-called Wulff shape centered at the origin. We put
$\kappa_N=|\mathcal W|$, and denote $\mathcal W_r=r\mathcal W$.

In the following, we often make use of some well-known properties of
$H$ and $H^o$:
\begin{gather*}
  H(D H^o(\xi))=H^o(D H(\xi))=1,\quad \forall \xi \in
  \R^N\setminus \{0\},
  \\
  H^o(\xi) D H(D H^o(\xi) ) = H(\xi) D H^o(D H(\xi) ) = \xi,\quad
  \forall \xi \in \R^N\setminus \{0\}. 
\end{gather*}

Let $\Omega$ be an open subset of $\mathbb R^N$. The total variation
of a function $u\in BV(\Omega)$ with respect to $H$ is (see \cite{ab}):
\[
\int_\Omega |Du|_H = \sup\left\{ \int_\Omega u\divergenza \sigma dx\colon
  \sigma \in C_0^1(\Omega;\R^N),\; H^o(\sigma)\le 1 \right\}.
\]
This yields the following definition of anisotropic perimeter of
$F\subset \R^N$ in $\Omega$:
\[
P_H(F;\Omega) = \int_\Omega |D\chi_F|_H= \sup\left\{ \int_F
  \divergenza \sigma dx\colon \sigma \in C_0^1(\Omega;\R^N),\;
  H^o(\sigma)\le 1 \right\}.
\]
The following co-area formula for the anisotropic perimeter
\begin{equation}\label{fr}
  \int_{\{u>t\}} H(Du) dx = \int_\Omega P_H (\{u>s\},\Omega)\, ds,\quad
  \forall u\in BV(\Omega)
\end{equation}
holds, moreover
\[
P_H(F;\Omega)= \int_{\Omega\cap \partial^*F} H(\nu_F) d\mathcal H^{N-1}
\]
where $\mathcal H^{N-1}$ is the $(N-1)-$dimensional Hausdorff measure
in $\mathbb R^N$, $\partial^*F$ is the reduced boundary of $F$ and
$\nu_F$ is the outer normal to $F$ (see \cite{ab}).

The anisotropic perimeter of a set $F$ is  finite if and only if the
usual Euclidean perimeter
\[
P(F;\Omega)=  \sup\left\{ \int_F \divergenza \sigma dx\colon
  \sigma \in C_0^1(\Omega;\R^N),\; |\sigma|\le 1 \right\}.
\]
is finite. Indeed, by properties \eqref{eq:omo} and \eqref{eq:lin} we
have that
\begin{equation}\label{eq:lin2}
  \frac{1}{c_2} |\xi| \le H^o(\xi) \le \frac{1}{c_1} |\xi|,
\end{equation}
and then
\begin{equation*}
  c_1 P(E;\Omega) \le P_H(E;\Omega) \le c_2 P(E;\Omega).
\end{equation*}
A fundamental inequality for the anisotropic perimeter is the
isoperimetric inequality
\begin{equation}\label{isop}
  P_H(E;\R^N) \ge N \kappa_N^{\frac 1 N} |E|^{1-\frac 1 N},
\end{equation}
which holds for any measurable subset $E$ of $\R^N$ (see for instance
\cite{aflt}).

Finally, if $u\in W^{1,1}(\Omega)$, then (see \cite{ab})
\[
\int_{\Omega} |Du|_H =\int_\Omega H(Du) dx.
\]

\subsection{Rearrangements, convex symmetrization and Lorentz spaces}
We recall some basic definition on rearrangements.
Let $\Omega$ be an bounded open set of $\R^N$,
$u:\Omega\rightarrow\R$ be a measurable function, and denote with
$|\Omega|$ the Lebesgue measure of $\Omega$.

The {distribution function} of $u$ is the map
$\mu_u:\mathbb R \rightarrow[0,\infty[$ defined by
\begin{equation*}
  \mu_u(t)\,=\,|\{x\in\Omega:|u(x)|>t\}|.
\end{equation*}
Such function is decreasing and right continuous.

The {decreasing rearrangement} of $u$ is the map
$u^*:\,[0,\infty[\rightarrow \R$ defined by
\begin{equation*}
  u^*(s):=\sup\{t\in\R:\mu_u(t)>s\}.
\end{equation*}
The function $u^*$ is the generalized inverse of $\mu_u$.

Following \cite{aflt}, the convex symmetrization of $u$ is the
function $u^\star(x)$, $x\in \Omega^\star$ defined by:
\begin{equation*}
  u^\star(x)=u^*(\kappa_N H^o(x)^N),
\end{equation*}
where $\Omega^\star$ is a set homothetic to the Wulff shape centered
at the origin having the same measure of $\Omega$, that is,
$\Omega^*=\mathcal W_R$, with $R=\big(\frac{|\Omega|}
{\kappa_N}\big)^{1/N}$.

The inequalities stated below will be useful in the next sections.
\begin{prop}
  Suppose $\lambda>0$, $1\le \gamma <+\infty$. Let $\psi$ a
  nonnegative measurable function on $(0,+\infty)$. The following
  inequalities hold:
 \begin{equation}\label{Hardy1d2}
    \int_0^{+\infty} \left(t^\lambda\int_t^{+\infty}
      \psi(s)ds\right)^\gamma \frac{dt}{t}
    \le \lambda^{-\gamma} \int_0^{+\infty}(t^{1+\lambda}\psi
    (t))^\gamma \frac{dt}{t}
  \end{equation}
  and
  \begin{equation}\label{Hardy1d1}
    \int_0^{+\infty} \left(t^{-\lambda}\int_0^t
      \psi(s)ds\right)^\gamma \frac{dt}{t}
    \le \lambda^{-\gamma}\int_0^{+\infty}(t^{1-\lambda}\psi
    (t))^\gamma \frac{dt}{t}.
  \end{equation}
\end{prop}

We recall that a measurable function $u:\Omega\rightarrow \R$ belongs
to the Lorentz space $L(p,q)$, $1 < p< +\infty$, if the quantity
\begin{equation*}
  \|u\|_{p,q}=\left\{
    \begin{array}{ll}
      \displaystyle\left\{\int_0^{+\infty} \left[t^{1/p} u^{*}(t)\right]^q
        \frac{dt}{t}\right\}^{1/q},&1\le q<+\infty,\\
      \displaystyle\sup_{0<t<+\infty} t^{1/p} u^{*}(t),&q=+\infty,
    \end{array}
  \right.
\end{equation*}
is finite.

In general $\|u\|_{{p,q}}$ is not a norm. As matter of fact, it is
possible to introduce a metric in $L(p,q)$ in the following
way. Let us define
\[
\|u\|_{(p,q)}=\|u^{**}\|_{p,q},
\]
with $u^{**}(t)=t^{-1}\int_0^t u^*(\sigma)\,d\sigma$.
We observe that also $u^{**}$ is a decreasing function, hence
$(u^{**})^*=u^{**}$. By means of the inequality
\eqref{Hardy1d1} and the properties of rearrangements, we
have that for $1<p< +\infty$ and $1\le q \le +\infty$,
\[
\|u\|_{p,q}\le\|u\|_{(p,q)}\le \frac{p}{p-1}\|u\|_{p,q}.
\]
Hence, the topology induced by $\|\cdot\|_{(p,q)}$ and
$\|\cdot\|_{p,q}$ is the same, that is
\begin{equation*}
u_{n}\rightarrow u \text { in }L(p,q) \iff \lim_n
\|u_n-u\|_{{p,q}}=0.
\end{equation*}

We stress that, for any fixed $p$, the Lorentz spaces $L(p,q)$
increase as the secondary exponent $q$ increases. Indeed, if
$1\le q\le r\le +\infty$, there exists a constant $C>0$ depending only
on $p,q$ and $r$ such that
\begin{equation*}
\|u\|_{p,r} \le C \|u\|_{p,q}.
\end{equation*}

More generally, the $L(p,q)$ spaces are related in the
following way:
\[
    L^r\subset L(p,1)\subset L(p,q)\subset L(p,p)=L^p\subset
    L(p,r)\subset L(p,\infty)\subset L^q,
\]
for $1 < q < p < r < +\infty$.

More details on Lorentz spaces can be found, for example, in
\cite{bs}. 

In the next sections, a basic tool will be the Hardy inequality,
stated below.
\begin{prop} For any $u\in W^{1,p}(\R^N)$,
\begin{equation}\label{Hardy}
\int_{\R^N} H(Du)^p dx \ge \Lambda_N \int_{\R^N}
\frac{|u|^p}{H^o(x)^p} dx,
\end{equation}
and the constant $\Lambda_N = \left(\frac{N-p}{p}\right)^p$ is
optimal, and not achieved.
\end{prop}
If $H(\xi)=|\xi|$, \eqref{Hardy} is the classical Hardy
inequality. For a general $H$, \eqref{Hardy} is proved in \cite{vs}.
\begin{rem}\label{remlor}
  The inequality \eqref{Hardy}, using the
  P\'olya Szeg\"o  inequality in the anisotropic case (see
  \cite{aflt}), can be rewritten as
\begin{equation*}
  \|u\|_{{p^*,p}} \le \frac{p}{(N-p)\kappa_N^{1/N}}
  \left(\int_{\R^N}H(Du)^p dx\right)^{1/p},
\end{equation*}
recovering the well-known result $W^{1,p}_0(\Omega) \subset
L(p^*,p)$ (see also \cite{asob}).
\end{rem}
\section{Statement of the problem and main results}
In this section we state the problem and the main results of the
paper. The proofs of the theorems are contained in sections 4 and 5.

Our aim is to prove a priori estimates and existence
results for problems of the type
\begin{equation}\label{eq:gen}
  \left\{
    \begin{array}{ll}
      -\divergenza{\left(a(x,u,Du)\right)}=b(x,u,Du)+g(x,u)+ f(x) &
      \text{in }\Omega, \\
      u = 0 & \text{on }\de\Omega,
    \end{array}
  \right.
\end{equation}
where $a\colon \Omega\times\R\times \mathbb R^N\rightarrow \mathbb
R^N$ is a Carath\'eodory function verifying
\begin{equation}\label{ellipt}
a(x,s,\xi) \cdot \xi \ge H(\xi)^p,
\end{equation}
with $1<p<N$, and
\begin{equation} \label{growth}
|a(x,s,\xi)|\le \alpha (|\xi|^{p-1}+|s|^{p-1}+k(x)), 
  \end{equation}
for a.e. $x\in\Omega$, for any  $(s,\xi)\in\R\times\R^N$, where
$\alpha>0$ and $k\in L_+^{p'}(\Omega)$. Moreover,
\begin{equation}\label{ip-mon}
(a(x,s,\xi) - a(x,s,\xi')) \cdot (\xi-\xi')>0,
\end{equation}
for a.e. $x\in\Omega$, for all $s\in \R,\xi\neq\xi'\in \R^N$.
As regards the lower order terms, we suppose that $b\colon
\Omega\times\R\times \mathbb R^N\rightarrow \mathbb R$ and $g\colon
\Omega\times\R\rightarrow \mathbb R$ are Carath\'eodory functions such
that
\begin{equation}\label{ipb}
  |b(x,s,\xi)|\le \beta(|s|) H(\xi)^q
\end{equation}
for a.e. $x\in\Omega$, for any  $(s,\xi)\in\R\times\R^N$, with
$p-1<q\le p$, and $\beta\colon[0,+\infty[\rightarrow
[0,+\infty[$ is continuous. Moreover,
\begin{equation}\label{ipg}
  \begin{array}{l}
    g(x,s)s \le c(x)|s|^p,\\
    |g(x,s)|\le d(x)|s|^{p-1},
  \end{array}
\end{equation}
for a.e. $x\in\Omega$ and for all $s\in \R$, where $c(x)$ and $d(x)$
are measurable functions in $\Omega$ such that
\begin{equation}\label{ip-c}
(c^+)^\star(x) \le \frac{\lambda}{H^o(x)^p},\quad\forall x\in
\Omega^\star,
\end{equation}
with $\lambda\ge 0$, and $d(x)\in L\left(\frac N p,\infty\right)$.

Finally, we take $f$ is in some suitable Lebesgue or Lorentz spaces
which will be specified in the following.

 If $f\in L((p^*)',p')$, we say that $u\in W^{1,p}_0(\Omega)$ is a
 weak solution of \eqref{eq:gen} if
\begin{equation*}
  \int_\Omega a(x,u,Du)\cdot D\varphi\, dx =
  \int_\Omega [ b(x,u,Du) + g(x,u) + f ]\varphi\, dx,
\end{equation*}
for any $\varphi\in W^{1,p}_0(\Omega)\cap L^{\infty}(\Omega)$.

The summability condition on $f$ given above is the
one which yields solutions in the energy space
$W_0^{1,p}(\Omega)$.

In order to state the main results, we need further assumptions on
$\beta$ and $\lambda$. First of all, let
  \begin{equation}
    \label{Binf}
    B(\infty)= \left(\frac{|\Omega|}{\kappa_N}\right)^{\frac{p-q}{N}}
    \left(\int_0^{\infty}
      \beta(t)^{\frac{1}{q-(p-1)}}\,dt\right)^{q-(p-1)} <\infty,
  \end{equation}
  and, for $1<m<\frac N p$, define the value $\lambda(m)$ as
      \begin{equation}
      \label{def-lambda}
      \lambda(m)=e^{-B(\infty)}\frac{N(m-1)(N-mp)^{p-1}}{m^p(p-1)^{p-1}}.
      \end{equation}
The first result we get is the following.
\begin{theo}
  \label{theo:deb}
  Suppose that \eqref{ellipt} $\div$ \eqref{ip-c}, \eqref{Binf}
  hold. Moreover, suppose that $0\le\lambda<\Lambda_Ne^{-B(\infty)}$. The
  following results hold:
  \begin{itemize}
  \item[(i)] if $f\in L((p^*)',p')$, problem \eqref{eq:gen} admits a
    weak solution $u\in W^{1,p}_0(\Omega)$;
  \item[(ii)]
    if $f\in L(m,\sigma)$, with  $(p^*)'< m < \frac N p$,
    $\max\left\{\frac{1}{p-1},1\right\} \le \sigma
    \le +\infty $, and $0\le \lambda<\lambda(m)$,
    then there exists a weak solution $u$ to \eqref{eq:gen} such that
    \[
    \|u\|_{\frac{Nm (p-1)}{N-mp},\,\sigma(p-1)}\le C
    \|f\|_{m,\sigma}^{\frac {1} {p-1}}.
    \]
  \end{itemize}
\end{theo}
From the embedding of Lorentz spaces, the above theorem gives
immediately the following result.
\begin{cor}
  Suppose that the hypotheses \eqref{ellipt} --
  \eqref{ip-c},\eqref{Binf} hold. If $f\in L^m(\Omega)$, with
  $(p^*)'< m < \frac N p$, and $0\le \lambda<\lambda(m)$, with
  $\lambda(m)$ as in \eqref{def-lambda}, then there exists a weak
  solution $u$ to \eqref{eq:gen} such that 
  \[
  \|u\|_{\frac{Nm(p-1)}{N-mp}}\le C,
  \]
  for some constant $C$ depending on the norm $\|f\|_m$.
\end{cor}
\begin{rem}
  \label{remopt}
  At least in the case $\beta=0$, the value $\lambda=\lambda(m)$, with
  $\lambda(m)$ as in \eqref{def-lambda}, is optimal in order to obtain
  the estimates in \eqref{stime1}. Let
  $(p^*)'=\frac{Np}{Np-N+p}<m<\frac{N}{p}$, and $0<\lambda <
  \Lambda_N=\left(\frac{N-p}{p}\right)^p$. For sake of
  simplicity, we prove the optimality of $\lambda(m)$ in the case of
  estimates in Lebesgue spaces, that is when $\sigma=\frac{Nm}{N-mp}$.

  We first consider radial solutions $v(x)=v(r)$, $r=H^o(x)$, $x\in
  \mathcal W$, of the equation
 \begin{equation*}
   -\mathcal{Q} v = \frac{\lambda}{H^o(x)^p} |v|^{p-2}v,\quad
   x\in \W,
 \end{equation*}
 where $\W$ is the Wulff shape. We also suppose that $H\in
 C^2(\R^N\setminus\{0\})$. In particular, we look for solutions 
 $v=v_\alpha=r^{-\alpha}$, with $\alpha>0$, which satisfy the ODE
  \begin{equation}\label{eq:radomo}
    \begin{array}{ll}
      -|v'|^{p-2}\left( (p-1)v''+\dfrac{N-1}{r} v' \right)=
      \dfrac{\lambda}{r^p}|v|^{p-2}v & \text{in }]0,1[.
    \end{array}
\end{equation}
Then, $v_\alpha$ solves \eqref{eq:radomo} if $\alpha$ satisfies the
equation
\begin{equation*}
F(\alpha)=\lambda, \quad\text{where }F(\alpha):=-(p-1)\alpha^p +
(N-p) \alpha^{p-1}.
\end{equation*}

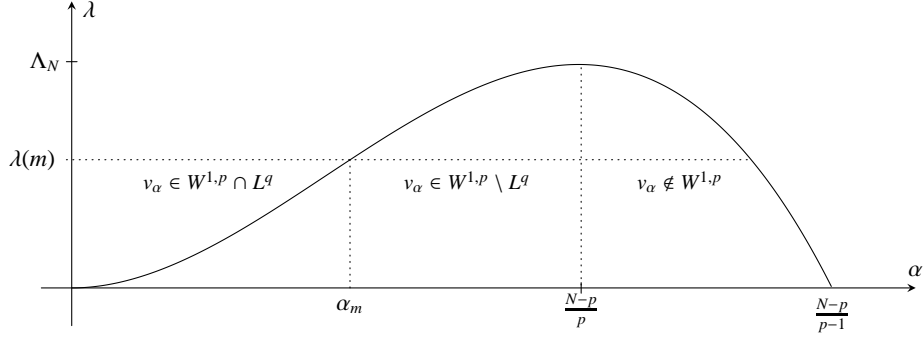
\begin{figure}
\definecolor{ttqqtt}{rgb}{0.2,0,0.2}
\begin{tikzpicture}[line cap=round,line join=round,>=triangle
  45,x=1.0cm,y=1.0cm,scale=10,,>=stealth]
\draw[->,color=black] (-0.04,0) -- (1.11,0);
\draw[->,color=black] (0,-0.05) -- (0,0.38);
\foreach \y in {0.3}
\draw[shift={(0,\y)},color=black] (.2pt,0pt) -- (-.2pt,0pt);
\draw (0,.3) node [anchor=east]
{\scriptsize$\Lambda_N$};
\clip(-0.1,-0.07) rectangle (1.12,0.381);
\draw[smooth,samples=200,domain=0.0:1.0] plot(\x,{0-2*(\x)^3+2*(\x)^2});
\draw (.67,-.2pt) -- (.67,.2pt) node [anchor=north] {\scriptsize
  $\frac{N-p}{p}$};
\draw[dotted] (.67,.292) -- (.67,0);
\draw (1,0) node [anchor=north] {\scriptsize $\frac{N-p}{p-1}$};
\draw (1.11,0) node [anchor=south] {\scriptsize $\alpha$};
\draw (0,0.37) node [anchor=west] {\scriptsize $\lambda$};
\draw[dotted] (-.2pt,.17) node [anchor=east]{\scriptsize $\lambda(m)$}
-- (.89,.17);
\draw[dotted] (.366,0) node [anchor=north]{\scriptsize $\alpha_m$}
-- (.366,.17);
\draw (.18, .17) node  [anchor=north]{\tiny $v_\alpha\in
  W^{1,p}\cap L^q$};
\draw (.52,.17) node  [anchor=north]{\tiny $v_\alpha\in
  W^{1,p}\setminus L^q$};
\draw (.8,.17) node  [anchor=north]{\tiny $v_\alpha\not\in
  W^{1,p}$};
\end{tikzpicture}
\caption{Graph of the function $F(\alpha)= -(p-1) \alpha^p + (N-p)
  \alpha^{p-1}$ in Remark \ref{remopt}. Here we 
  consider the case $\frac N p> m>(p^*)'$, and $q=\frac{Nm(p-1)}{N-mp}$.}
\end{figure}

We stress that $v_\alpha\in L^{\frac{Nm(p-1)}{N-mp}}(\W)$ if
$\alpha<\alpha_m=\frac{N-mp}{m(p-1)}$. Moreover, $\alpha_m < \frac
{N-p}{p}$, being $m>(p^*)'$. Hence $v\in W^{1,p}(\mathcal W)$.
Furthermore, $F(\alpha_m)=\lambda(m)$. In order to prove the
optimality of $\lambda(m)$, we observe that the positive function
$z_\alpha(x) =z_\alpha(r) = v_\alpha(r)-1$ is such that
\begin{equation*}
  \left\{
    \begin{array}{ll}
      -\mathcal Q z_\alpha = \dfrac{\lambda(m)}{H^o(x)^p}
      z_\alpha^{p-1} + g(H^o(x)) &
      \text{in }\mathcal W, \\[.2cm]
       z_\alpha = 0 & \text{on }\de \mathcal W,
    \end{array}
  \right.
\end{equation*}
with
\[
g(r)=\lambda(m) \dfrac{ (z_\alpha+1)^{p-1}-z_\alpha^{p-1} }{r^p} =
\lambda(m)
\frac{1-(1-r^{\alpha_m})^{p-1}}{r^{\alpha_m}} \cdot \frac
{1}{r^{p+\alpha_m(p-2)}}.
\]
The condition $m<\frac{N}{p}$ gives that $g \in
L^{m}(\W)$. Nevertheless, for $\alpha\ge \alpha_m$,
$z_\alpha$ does not belong to $L^{\frac{Nm(p-1)}{N-mp}}(\W)$. This
prove the optimality of $\lambda(m)$.
\end{rem}

Next step is to state an existence and regularity result for problems
whose datum $f$ is in $L^m$, $m>1$. To this aim, we deal with
entropy solutions. 

Following \cite{bbggpv}, for a general $f\in
L^1(\Omega)$ we will say that a measurable function $u$ is an entropy
solution of \eqref{eq:gen} if $g(x,u),b(x,u,Du) \in L^1(\Omega)$ and, for any
$k>0$, $T_k(u)\in W_0^{1,p}(\Omega)$ and
\begin{equation}\label{defsol-entr}
  \int_\Omega a(x,u,Du)\cdot D T_k(u-\varphi) \, dx \le
  \int_\Omega [ b(x,u,Du) + g(x,u) + f ] T_k ( u - \varphi ) \, dx,
\end{equation}
for all $\varphi\in W_0^{1,p}(\Omega)\cap L^{\infty}(\Omega)$.
When $T_k(u)\in W_0^{1,p}(\Omega)$ for any $k>0$, it is possible to
define the weak gradient of $u$, namely $Du$, as the function such
that $DT_k(u)=(Du)\chi_{\{|u|\le k\}}$, for any $k>0$ (see \cite{bbggpv}).

The following result holds.
\begin{theo}
  \label{theo:entr}
Let us suppose that \eqref{ellipt} -- \eqref{ip-c}, \eqref{Binf}
hold. If $f\in L(m,\sigma)$, $1<m < (p^*)'$, $p' \le \sigma\le
\infty$, and $0\le \lambda<\lambda(m)$, with $\lambda(m)$ as in
\eqref{def-lambda}, then there exists an entropy
solution of \eqref{eq:gen} such that
\begin{equation}
  \label{estg}
  \norm{H(Du)^{p-1}}_{\frac{Nm}{N-m},\sigma} \le C\|f\|_{m,\sigma}.
\end{equation}
\end{theo}
\begin{rem}
 We stress that the assumptions of Theorem \ref{theo:entr} do not
 allow to obtain the estimate \eqref{estg} for $\sigma=m$. As matter
 of fact, defined $\bar m= \frac{N}{Np-N+1}$, when $f\in L^m(\Omega)$,
 $\max\{1,\bar m\}\le m<(p^*)'$, it is possible to prove the existence of
 a solution such that  
 \begin{equation*}
 \|H(Du)\|_{\frac{Nm(p-1)}{N-m}} \le C(\|f\|_{m}).
 \end{equation*}
  We refer the reader to remarks \ref{remstime} and \ref{remex}.

 In the case $1<m<\max\{1,\bar m\}$, the solutions we obtain no
 longer belong to a Sobolev space, but Theorem \ref{theo:entr}
 guarantees that there exists a solution $u$ such that, for example,
 \[
  \norm{H(Du)^{p-1}}_{\frac{Nm}{N-m},\infty} \le C\|f\|_{m,\infty}.
 \]
 Actually, $\bar m>1$ only if $p< 2-\frac 1 N$. In this case an
 estimate of the type
 \[
  \|H(Du)^{p-1}\|_{\frac{Nm}{N-m}} \le C(\|f\|_{m})
  \]
  holds for any $1<m\le \bar m$ (see Remark \ref{remstime}).
\end{rem}

\begin{rem}
  We explictly observe that, in general, the above Theorem does not
  hold for $f\in L^1(\Omega)$. For example, it has been proved in
  \cite{bop} that, when $H(\xi)=|\xi|$, $p=2$ and $\beta\equiv 0$, for
  any $\lambda>0$ no a priori estimate holds for problem
  \eqref{eq:gen}. As matter of fact, when $\lambda=0$, if $\beta\in
  L^1(\Omega)$ and $p=q$ it is possible to prove the existence of a
  solution of \eqref{eq:gen} (see for example \cite{seg03},
  \cite{por02}).
\end{rem}

\begin{rem}\label{opth}
  We stress that the bounds \eqref{eq:lin} and \eqref{eq:lin2} on $H$
  and $H^o$, and the conditions
  \eqref{ellipt}, \eqref{ipb} and \eqref{ip-c} give that
  \[
  a(x,s,\xi)\cdot \xi \ge c_1^p |\xi|^p,\quad |b(x,s,\xi)| \le
  c_2^q\,\beta(|s|)|\xi|^q,
  \]
  and
  \[
  (c^+)^\star(x) \le \frac{\lambda c_2^p}{|x|^p}.
  \]
  Hence, under the above growth conditions, the classical Schwarz
  symmetrization tecnique can be applied to problem \eqref{eq:gen}. In
  this way, it is possible to obtain analogous results than
  Theorem \ref{stime1}, and consequently Theorem \ref{theo:deb}, but
  requiring a stronger assumption on the smallness of $\lambda>0$ (see
  also \cite{aflt} and Remark 4.1 in \cite{dpg3}). This justifies the
  use of the more general convex symmetrization.
\end{rem}

\begin{rem}
  Let $H(\xi)=|\xi|$. If $q=p$, the regularity estimates obtained in
  Theorems \ref{stime1} and \ref{stime2g} are slightly more general
  than the analogous one contained in \cite{lmp}. Indeed, in such
  paper the datum $f$ in suitable Lebesgue space is considered, while
  we give optimal regularity results in Lorentz spaces.
\end{rem}

\section{A priori estimates and approximating problems}

The first aim of this section is to prove three
integro-differential inequalities for the rearrangements of solutions of
\eqref{eq:gen}, in the spirit of the comparison results contained, for
instance, in \cite{tal1}, \cite{ta79}, \cite{aftgrad},
\cite{fermes}. To prove such inequalities we need the additional
assumption \eqref{Binf}.

\begin{theo}\label{stima0}
  Suppose that \eqref{ellipt}, \eqref{growth}, \eqref{ipb} $\div$
  \eqref{ip-c} hold, and $f\in L((p^*)',p')$. Moreover, suppose that
  $\beta(s)$ verifies \eqref{Binf}. Then any weak solution $u\in
  W_0^{1,p}(\Omega)$ of problem \eqref{eq:gen} satisfies
\begin{equation}\label{eq:est-grad}
- \frac{d}{dt}\int_{\{|u|>t\}} H(Du)^p dx\le e^{B(\infty)}
\int_0^{\mu_u(t)} [(c^+)^*(s)u^*(s)^{p-1}+ f^*(s)]ds,\quad \text{a.e. }t>0,
\end{equation}
  \begin{equation}\label{equ}
    u^*(s)\le
    e^{\frac{B(\infty)}{p-1}}\left(N\kappa_N^{1/N}\right)^{-p'}
   \int_s^{|\Omega|} t^{-\frac{p'}{N'}}
    \left(\int_0^t [(c^+)^*(r)(u^*(r))^{p-1} +f^*(r)]
      dr\right)^{\frac{1}{p-1}} dt,\quad s\in ]0,|\Omega|].
  \end{equation}
  Moreover, for any $\alpha>\frac{p'}{N'}-1$ we have that
  \begin{multline}
    \label{eq:gradlor}
    [H(Du)^*(s)]^p \le \\ \le e^{\frac{B(\infty)}{p-1}}
    \left(N\kappa_N^{1/N}\right)^{-p'} \left[
      \frac{1}{s^{\alpha+1}}
      \int_0^s t^{\alpha -\frac{p'}{N'}}
      \left(\int_0^t [(c^+)^*(r)(u^*(r))^{p-1}+f^*(r)] dr \right)^{p'} d t
      + \right. \\ \left. +
      \frac 1 s
\int_s^{|\Omega|} t^{-\frac{p'}{N'}}
      \left(\int_0^t [(c^+)^*(r)(u^*(r))^{p-1}+f^*(r)] dr \right)^{p'} d t
    \right],\quad s\in ]0,|\Omega|[.
  \end{multline}
\end{theo}
\begin{proof}
  Let $u\in W^{1,p}_0(\Omega)$ be a solution to
  \eqref{eq:gen}. Using the following test function $\varphi\in
  W^{1,p}_0(\Omega)$,
\[
\varphi(x)=
\begin{cases}
 0 & |u|\le t,\\
(|u|-t)\sign u & t < |u| \le t+h, \\
 h\sign u  & t+h < |u|,
\end{cases}
\]
by the hypotheses \eqref{ellipt}, \eqref{ipb} $\div$
  \eqref{ip-c}, and the Hardy-Littlewood inequality we obtain
\begin{multline*}
-\frac{d}{dt}\int_{\{|u|>t\}}H(Du)^p \,dx \le \\ \le \int_{\{|u|>t\}}
\beta(|u|) H(Du)^q\,dx + \int_0^{\mu_u (t)}
\left((c^+)^*(\sigma)u^*(\sigma)^{p-1} +f^*(\sigma)\right)d\sigma.
\end{multline*}
By the continuity of $\beta$ we have
\[
\int_{\{|u|>t\}}
\beta(|u|) H(Du)^q\,dx =\int_t^{+\infty} \beta(s) \left( -\frac{d}{ds}
\int_{\{|u|>s\}}H(Du)^q\, dx\right)\, ds.
\]
Hence, using also the H\"older inequality we get
\begin{equation}
  \label{eq:inbeta}
\int_{|u|>t} \beta(|u|) H(Du)^q dx \le \int_t^{+\infty} \beta(s)
\left[
  \left(
    -\frac{d}{ds} \int_{|u|>s} H(D u)^p dx
  \right)^{q/p}
  (-\mu'_{u}(s))^{1-q/p}
\right] d s,
\end{equation}
and
\begin{multline*}
    \left( -\frac{d}{ds}\int_{\{|u|>t\}} H(Du)^p dx \right)^{\frac{q}{p}}
    (-\mu_u'(s))^{1-\frac q p}\le \\
\leq    \left( -\frac{d}{ds}\int_{\{|u|>t\}} H(Du) dx
    \right)^{q-p} \left( -\frac{d}{ds}
      \int_{\{|u|>s\}} H(Du)^p dx \right)(-\mu_u'(s))^{p-q}.
  \end{multline*}
  The coarea formula \eqref{fr} and the isoperimetric inequality
\eqref{isop} imply
\begin{multline}
  \label{mul:1}
  \left( -\frac{d}{ds}\int_{\{|u|>s\}} H(Du)^p dx \right)^{\frac{q}{p}}
  (-\mu_u'(s))^{1-\frac q p}\le \\ \le \left(N\kappa_N^{1/N}
    \mu_u(s)^{1-\frac 1 N}\right)^{q-p}
  \left( -\frac{d}{ds}\int_{\{|u|>s\}} H(Du)^p dx \right)
  (-\mu_u'(s))^{p-q}.
\end{multline}
So, from (\ref{eq:inbeta}) and (\ref{mul:1}) we have
\begin{multline}\label{eq:10}
  -\frac{d}{dt}\int_{\{|u|>t\}}H(Du)^p \,dx \le \\ \le
  \left(N\kappa_N^{1/N} \right)^{q-p}
  \int_t^{+\infty} \beta(s) \left( -\frac{d}{ds}
    \int_{\{|u|>s\}}H(Du)^p\, dx\right)
  \left(\frac{-\mu_u'(s)}{\mu_u(s)^{1-\frac 1 N}}\right)^{p-q}
    ds + \int_0^{\mu_u (t)} z(s)\,ds,
\end{multline}
where for sake of brevity we set $z(s)=(c^+)^*(s)u^*(s)^{p-1} +f^*(s)$.

Now, using the Gronwall Lemma and the properties of rearrangements in
\eqref{eq:10}, it follows that
\begin{equation}
\label{Grom}
-\frac{d}{dt}\int_{\{|u|>t\}}H(Du)^p \,dx \le  \int_0^{\mu_u(t)}
z(s) \exp\left\{  \left(N\kappa_N^{1/N} \right)^{q-p} \int_t^{u^*(s)}
\beta(y) \left(\frac{-\mu_u'(y)}{\mu_u(y)^{1-\frac 1 N}}\right)^{p-q} dy
\right\} ds.
\end{equation}
On the other hand, if $p-1<q<p$, using H\"older inequality we have
\begin{equation}\label{mul:4}
    \int_t^{u^*(s)}
    \beta(y) \left( \frac{-\mu_u'(y)}
    {\mu_u(y)^{1-\frac{1}{N}}}\right)^{p-q}dy \le \left[\int_t^{u^*(s)}
    \beta(y)^{\frac{1}{1-p+q}}dy\right]^{1-p+q}
    \left[ \int_t^{u^*(s)}\frac{-\mu_u'(y)}
    {\mu_u(y)^{1-\frac{1}{N}}}dy \right]^{p-q}.
\end{equation}
(Observe that last inequality is trivial if $q=p$). Furthermore, by
the properties of the distribution function $\mu$ of $u$, we have
\begin{equation}\label{equ:3}
\int_t^{u^*(s)}\frac{-\mu_u'(y)}
    {\mu_u(y)^{1-\frac{1}{N}}}dy \leq \int_0^{+\infty}\frac{-\mu_u'(y)}
    {\mu_u(y)^{1-\frac{1}{N}}}dy \leq N|\Omega|^{\frac 1 N}.
  \end{equation}

Using (\ref{mul:4}) and (\ref{equ:3}) in (\ref{Grom}), we get
\begin{multline}\label{mul:5}
- \frac{d}{dt}\int_{\{|u|>t\}} H(Du)^p dx \le
\int_0^{\mu(t)} z(s) \exp{
     \left\{\left(\frac{|\Omega|}{\kappa_N}\right)^{\frac{p-q}{N}}
      \left[\int_t^{u^*(s)}
    \left(\beta(y)\right)^{\frac{1}{1-p+q}}dy\right]^{1-p+q}
\right\}}ds \le \\
\le e^{B(\infty)}
\int_0^{\mu_u(t)} [(c^+)^*(s)u^*(s)^{p-1}+ f^*(s)]ds,
\end{multline}
where last inequality follows by \eqref{Binf}, and $B(\infty)$ is
finite by the assumption on $\beta$. This proves the inequality
\eqref{eq:est-grad}.

In order to show \eqref{equ}, using similarly as before the H\"older
inequality, the coarea formula and the isoperimetric inequality in the
left-hand side of \eqref{mul:5}, we get that
\[
(-\mu_u'(t))^{-1} \le e^{\frac{1}{p-1}B(\infty)}
\left(N\kappa_N^{1/N}\right)^{-p'}
\mu_u(t)^{-\frac {p'}{N'}}
\left(\int_0^{\mu_u(t)} [(c^+)^*(s)u^*(s)^{p-1}+ f^*(s)]ds\right)^{1/(p-1)}.
\]
Integrating between $s$ and $|\Omega|$, we get \eqref{equ}.

Finally, following the argument contained in \cite{aftgrad}, we get
last inequality \eqref{eq:gradlor}.
\end{proof}

In order to get existence and regularity results for \eqref{eq:gen},
we will consider the approximated problems 
  \begin{equation}\label{eq:gen-app}
  \left\{
    \begin{array}{ll}
      -\divergenza{a(x,u_n,Du_n)}=b_n(x,u_n,Du_n) + g_n(x,u_n) +
      f_n(x) & \text{in }\Omega, \\
      u_n = 0 & \text{on }\de\Omega,
    \end{array}
  \right.
\end{equation}
where $b_n(x,s,\xi)=T_n(b(x,s,\xi))$, $g_n(x,s)=T_n(g(x,s))$,
$f_n(x)=T_n(f(x))$, and $T_n(s)=\max\{-n,\min\{s,n\}\}$ is the standard 
truncature function. Under the assumptions \eqref{ellipt} $\div$
\eqref{ip-c}, the existence of a weak solution $u_n \in
W_0^{1,p}(\Omega)$ to problem \eqref{eq:gen-app} follows by the
classical Leray-Lions result (see \cite{ll}). Moreover, such solution
is bounded. 

Now we use the inequalities proved in Theorem \ref{stima0} in order to
obtain some a priori estimates for problems 
\eqref{eq:gen-app}. As stated in the introduction, an additional
assumption on the smallness of the value $\lambda$, depending on the
summability of $f$, is needed.
\begin{theo}\label{stime1}
Suppose that the hypotheses \eqref{ellipt} $\div$ \eqref{ip-c},
\eqref{Binf} hold. Let $f\in L(m,\sigma)$, with  $1< m < \frac N
p$ and $\max\left\{\frac{1}{p-1},1\right\} \le \sigma \le +\infty $,
and
\begin{equation*}
0\le \lambda < \lambda(m),
\end{equation*}
with $\lambda(m)$ as in \eqref{def-lambda}.
Then, the weak solutions $u_n$ of \eqref{eq:gen-app}
are such that
\begin{equation}
  \label{reg}
\norm{ |u_n|^{p-1} }_{\frac{Nm}{N-mp},\,\sigma}\le C \|f\|_{m,\sigma},
\end{equation}
for some positive constant $C$ independent of $n$.
\end{theo}
\begin{proof}
  We first consider the case $\sigma=+\infty$.
  Problem \eqref{eq:gen-app} verifies the assumptions of Theorem
  \ref{stima0}. Hence, we can use the inequality \eqref{equ} for
  $u_n$. Recalling that
  $(c^+)^*(\tau)\le \lambda \kappa_N^{\frac p N}
  \tau^{-\frac p N}$, $\tau\in ]0,|\Omega|[$, we obtain that
  \begin{multline}
    \label{eqm:3}
    u_n^*(s)^{p-1} \le e^{B(\infty)}(N\kappa_N^{\frac 1 N})^{-p}
    \left\{
      \int_s^{|\Omega|} t^{\frac {-p'}{N'}}
      \left[
        \int_0^t \left(
        \lambda \kappa_N^{\frac p N} \tau ^{-\frac p N}
        (u_n^*(\tau))^{p-1} +f_n^*(\tau)\right) d\tau
      \right]^{\frac{1}{p-1}} dt
    \right\}^{p-1}
    \le
    \\
    \le
     e^{B(\infty)}(N\kappa_N^{\frac 1 N})^{-p}
    \left\{
      \int_s^{|\Omega|} t^{\frac {-p'}{N'}}
      \left[
        \int_0^t \left(
        \lambda \kappa_N^{\frac p N}
        \|u_n^{p-1}\|_{\frac{Nm}{N-mp},\infty}
        +\|f_n\|_{m,\infty}\right)
      \tau^{-\frac 1 m}
      d\tau \right]^{\frac{1}{p-1}} dt
    \right\}^{p-1}.
    \end{multline}
Last inequality follows simply by the definition of the Lorentz
norms.

Hence, from \eqref{eqm:3}, recalling also that $|f_n|\le |f|$, we get
\begin{multline*}
  u^*_n(s)^{p-1} \le \left(
    \lambda e^{B(\infty)}N^{-p}
    \|u_n^{p-1}\|_{\frac{Nm}{N-mp},\infty}
    + C \|f\|_{m,\infty} \right) \left[\int_s^{|\Omega|} t^{-\frac{p'}{N'}}
    \left(\int_0^t \tau^{-\frac 1 m} d\tau
    \right)^{\frac{1}{p-1}} dt \right]^{p-1} \le \\ \le \left( \frac
    {\lambda}{\lambda(m)} \|u_n^{p-1}\|_{\frac{Nm}{N-mp},\infty} +
    C \|f\|_{m,\infty} \right) s^{-\frac{N-mp}{Nm}},
\end{multline*}
where $C$ is a constant which does not depend on $n$. Finally, being
$\lambda<\lambda(m)$, the above inequality gives that
\[
 \norm{ |u_n|^{p-1}}_{\frac{Nm}{N-mp},\infty} \le C \|f\|_{m,\infty},
\]
and we get the thesis when $\sigma=\infty$. Now, suppose that
$\max\left\{\frac{1}{p-1},1\right\} \le\sigma<+\infty$. For the sake of brevity, we denote with $z(\tau)$
the function $z(\tau)=
(c^+)^*(\tau)(u_n^*(\tau))^{p-1}+f_n^*(\tau)$. As before, from
\eqref{equ} applied to $u_n$ we obtain that
\begin{multline}\label{eq:dim1}
  \norm{ |u_n|^{p-1}}^\sigma_{\alpha,\sigma}
  =\int_0^{|\Omega|} s^{\frac {\sigma}{
      \alpha} } u_n^*(s)^{\sigma(p-1)}
  \frac{ds}{s} \le \\ \le
  e^{\sigma\, B(\infty)} \left(N\kappa_N^{1/N}\right)^{-p\sigma}
  \int_0^{|\Omega|}
  \left(s^{\frac {1} {\alpha(p-1)} } \int_s^{|\Omega|}
    t^{-\frac{p'}{N'}} \left(\int_0^t
      z(\tau) d\tau\right) ^{\frac{1}{p-1}}
    \,d t \right)^{\sigma(p-1)} \frac{ds}{s}  \le \\ \le
  e^{\sigma\,B(\infty)}\left(N\kappa_N^{1/N}\right)^{-p \sigma}
  [\alpha(p-1)]^{\sigma(p-1)}
  \int_0^{|\Omega|} \left[s^{
      \frac 1 \alpha -1 +\frac{p}{N} }
    \int_0^s
    z(\tau) d\tau \right]^{\sigma}
  \frac { ds }{ s },
\end{multline}
where last inequality is obtained by using \eqref{Hardy1d2}, being
$\sigma(p-1) \ge 1 $.

Let us observe that
\begin{equation}
  \label{eq:dimneg}
  \frac 1 \alpha -1 +\frac{p}{N} < 0 \iff
  \alpha>\frac{N}{N-p}.
\end{equation}
If this is the case, being $\sigma\ge 1$, by \eqref{Hardy1d1}
we get from \eqref{eq:dim1} that
\begin{multline}\label{dim:K}
  \norm{ |u_n|^{p-1}}_{\alpha,\sigma}^\sigma\le \\ \le e^{\sigma\,B(\infty)}
  (N\kappa_N^{1/N})^{-p\sigma}[\alpha(p-1)]^{\sigma(p-1)}
  \left( 1-\frac{p}{N} -\frac 1 \alpha
  \right)^{-\sigma} \int_0^{|\Omega|}\left[ s^{\frac
      {1}{\alpha}+\frac p N} z(s)\right]^{
    {\sigma}}\frac{ds}{s}=\\= K \int_0^{|\Omega|}\left[ s^{\frac
       {1}{\alpha}+\frac p N} z(s)\right]^{
     {\sigma}}\frac{ds}{s}.
 \end{multline}
 Hence, using the Minkowski inequality, we get that
 \begin{multline}\label{dim:3}
   \frac{1}{K^{\frac{1}{\sigma}}} \norm{|u_n|^{p-1}}_{\alpha,\sigma} \le
   \left(\int_0^{|\Omega|} \left[ s^{\frac{1}{\alpha}+\frac p N}
       z(s)\right]^{\sigma}
     \frac{ds}{s}\right)^{\frac 1 \sigma} \le \\
   \left(\int_0^{|\Omega|} \left(s^{
         \frac{p}{N}}
       (c^+)^*(s)\right)^{{\sigma}} s^{\frac \sigma \alpha} u_n^*(s)^{\sigma(p-1)}
     \,\frac{ds}{s} \right)^{\frac 1 \sigma}
   + \left( \int_0^{|\Omega|} \left(s^{ \frac{1}{\alpha}+\frac {p}{N}}
       f_n^*(s)\right)^\sigma \frac{ds}{s} \right)^{\frac 1 \sigma}.
 \end{multline}
 Being $(c^+)^*(s)\le \lambda \kappa_N^{\frac{p}{N}} s^{-\frac{p}{N}}$,
 writing explicitly the value of $K$ in \eqref{dim:K}, 
 \eqref{dim:3} implies
 \[
 \kappa_N^{\frac p N}\left(\frac{ \alpha(N-p)-N }
   { e ^{B(\infty)}  N^{1-p}  \alpha^p(p-1)^{p-1} }
   -\lambda\right)
 \norm{|u_n|^{p-1}}_{\alpha,\sigma} \le
 \|f_n\|_{\frac{N\alpha}{N+\alpha p}, \sigma }\,.
 \]
 Hence, for $m=\frac{N\alpha}{N +\alpha p}$ we have that
 $\alpha=\frac{Nm}{N-mp}$ verifies \eqref{eq:dimneg}, and for
 \[
 \lambda < e^{-B(\infty)}\frac {\alpha(N-p)-N}{N^{1-p}\alpha^p(p-1)^{p-1}} =
 e^{-B(\infty)}N \frac{(N-mp)^{p-1}(m-1)}{m^p(p-1)^{p-1}}=\lambda(m),
 \]
 we get
 \[
 \norm{|u_n|^{p-1}}_{\frac{Nm}{N-mp},\sigma}\le C\|f_n\|_{m,\sigma},
 \]
 for some constant $C$. Being $|f_n|\le |f|$, we get the thesis.
\end{proof}

\begin{rem}
  \label{remstime}
  We observe that, in particular, the result obtained in Theorem
  \ref{stime1} provides estimates in terms of suitable Lebesgue norms
  of $u_n$, and $f$. Indeed, choosing $\sigma=
  \frac{Nm}{N-mp}$ in \eqref{reg}, and supposing that
  $\frac{Nm}{N-mp}\ge \max\{\frac{1}{p-1},1\} $, being 
  $L^m(\Omega)\subset L\left(m,\frac{Nm}{N-mp}\right)$, if
  $\lambda<\lambda(m)$ we have that 
  \begin{equation}\label{regrem}
  \left\| |u_n|^{p-1} \right\|_{\frac{Nm}{N-mp}} \le C\|f\|_m.
  \end{equation}
  Clearly, if $p\ge 2$ no further assumption on $m\in\left]1,\frac N
    p\right[$ 
  is needed to get \eqref{regrem}. Otherwise, we have to require that
  $m\ge \frac{(p^*)'}{p}$. This additional hypothesis is due only to
  technical reasons, but, when $\lambda<\lambda(m)$, the estimate
  \eqref{regrem} holds also in the case $1<m<\frac{(p^*)'}{p}$. For sake
  of completeness, we sketch the proof of \eqref{regrem} in the general
  case. We use the same notation of Theorem \ref{stime1}.

  Let $\eps>0$, and $\alpha>0$. By \eqref{mul:5} we have:
  \[
  -\frac{d}{dt} \int_{\{|u_n|> t\}} \frac{H(Du_n)^p}{(\eps+|u_n|)^\alpha}
  dx \le e^{B(\infty)}(1+t)^{-\alpha} \int_0^{\mu_{u_n}(t)} [(c^+)^*
  u_n^*(\tau)^{p-1} +f^*(\tau)]\,
  d\tau,
  \]
  and
  \begin{equation}
    \label{hardyeps}
  \int_\Omega \frac{H(Du_n)^p}{(\eps+|u_n|)^\alpha} dx \le e^{B(\infty)}
  \frac{1}{1-\alpha} \int_0^{|\Omega|}
  [(\eps+u_n^*(s))^{1-\alpha}-\eps^{1-\alpha}] [(c^+)^* u_n^*(s)^{p-1} +f^*(s)]ds.
  \end{equation}
  Now we recall that for any $\eps>0$ sufficiently small and $0<\gamma<1$,
  the following inequality holds:
  \[
  x^{p-1}[(\eps+x)^{p\gamma-(p-1)}-\eps^{p\gamma-(p-1)}] \le
  [(\eps+x)^{\gamma}-\eps^{\gamma}]^p, \quad \forall x\ge 0.
  \]
  Then we have, for $0<\alpha<1$,
  \begin{multline}\label{hardyeps2}
  \int_0^{|\Omega|} (c^+)^* u_n^*(s)^{p-1} [(\eps+u_n^*
  (s))^{1-\alpha}-\eps^{1-\alpha}] 
  ds \le \\ \le \lambda \kappa_N^{\frac p N} \int_0^{|\Omega|}
  \frac{[(\eps+u_n^*(s))^{1-\frac{\alpha}{p}}-\eps^{1-\frac{\alpha}{p}}]^p}{s^{\frac
      p N}} ds = \lambda \kappa_N^{\frac p N} \| (\eps+|u_n|)^{1-\frac
  \alpha p} -\eps ^{1-\frac \alpha p} \|^p_{p^*,p}\;.
  \end{multline}
  Moreover,
  \begin{equation}
    \label{hardyeps3}
    \int_0^{|\Omega|} [(\eps+u_n^*(s))^{1-\alpha}-\eps^{1-\alpha}]
    f^*(s)\,ds \le C \|f\|_m \|
    |u_n|^{1-\alpha}\|_{m'}.    
  \end{equation}

  As matter of fact, we have that by Hardy inequality
  \eqref{Hardy}, 
  \[
  \int_\Omega \frac{H(Du_n)^p}{(\eps+|u_n|)^\alpha} dx =
  \left(\frac{p}{p-\alpha}\right)^p \int_\Omega
  H(D((\eps+|u_n|)^{1-\frac \alpha p} ) )^p dx\ge
  \left(\frac{N-p}{p-\alpha}\right)^p \|
  ((\eps+|u_n|)^{1-\frac \alpha p}-\eps^{1-\frac \alpha p} )\|^p_{p^*,p}.
  \]
  Using the above inequality, \eqref{hardyeps2} and
  \eqref{hardyeps3} in \eqref{hardyeps} we have that, by the
  properties of rearrangements and the Fatou lemma,
  \[
  \left(  \left(\frac{N-p}{p-\alpha}\right)^p
    -\frac{\lambda}{1-\alpha} e^{B(\infty)} \right)
  \|u_n^{1-\frac \alpha p}\|_{p^*,p}^p \le C \|f\|_m \|
  u_n^{1-\alpha}\|_{m'}.
  \]
  Let choose $\alpha$ such that $(1-\alpha)m'=(1-\alpha/p)p^*$, 
  after some computations we get that, being $m<\frac{(p^*)'}{p}$, then
  $0<\alpha<1$ and
  \[
  \left(\lambda(m)-\lambda \right)\||u_n|^{p-1}\|_{\frac{Nm}{N-mp}} 
  \le C \left(\|f\|_m\right),
  \]
  and for $\lambda<\lambda(m)$ we get the estimate
  \eqref{regrem}.

  Finally, the above estimate gives also a uniform bound for $Du_n$, that is
  \begin{equation}
    \label{eeee}
  \| H(Du_n)^{p-1} \|_{\frac{Nm}{N-m}} \le C(\|f\|_m).
  \end{equation}
  Clearly, if $m>\max\{1,\bar m\}$, $\bar m=\frac{N}{N(p-1)+1}$,
  this follows from \eqref{regrem} by Sobolev inequality. Otherwise,
  the above computations give that for $\eps>0$
  \[
  \int_\Omega \frac{H(Du_n)^{p}}{(\eps+|u_n|)^\alpha} dx \le C \|f\|_m
  \|(\eps+|u_n|)^{1-\alpha}\|_{m'}.
  \]
  Hence, reasoning as in \cite{lmp}, H\"older and Sobolev
  inequalities give \eqref{eeee}.
  
\end{rem}

\begin{rem} We explicitly observe that if
  $m=(p^*)'=\frac{Np}{Np-N+p}$, then
  \[
  \lambda(m)=\left(\frac{N-p}{p}\right)^p e^{-B(\infty)}=\Lambda_N
  e^{-B(\infty)}.
  \]
\end{rem}
Next proposition will be an useful tool to pass to the limit in the
approximating problems \eqref{eq:gen-app}, and is a consequence of the
obtained estimates on $u_n$.
\begin{prop}
  \label{prop:bound}
  Under the hypothesis of Theorem \ref{stime1}, for any $t>0$ it holds
  that
  \begin{equation}
    \label{eq:zero}
    \int_0^{|\Omega|} [(c^+)^*(s)u_n ^*(s)^{p-1} +f_n^*(s)] ds \le C
    \|f\|_{m,\sigma},
  \end{equation}
  and
  \begin{equation}
    \label{eq:uno}
  \int_{\{|u_n|>t\}} |b_n(x,u_n,Du_n)| dx \le  C
  t^{-\alpha},
  \end{equation}
  where $\alpha = \frac{(p-1)[N(m-1)+m(p-q)]}{N-mp}>0$, and $C$ denotes a
  positive constant independent of $n$ and $t$.
\end{prop}
\begin{proof}
The estimate \eqref{eq:zero} follows immediately from \eqref{reg} and
the definition of Lorentz space.

In order to show \eqref{eq:uno}, let $t>0$. Reasoning as in Theorem
\ref{stima0}, and using the same notation, we have that
  \begin{multline}
    \label{mul:betaest}
    \int_{\{|u_n|>t\}} |b_n(x,u_n,Du_n)| dx \le
    \int_{\{|u_n|>t\}} \beta(|u_n|)H(Du_n)^q dx \le \\ \le
    \int_t^{+\infty} \beta(s) \left[
      \left(- \frac{d}{ds} \int_{\{|u_n|>s\}} H(Du_n)^p dx
      \right)^{\frac q p} (-\mu_{u_n}'(s))^{1-\frac q p}
    \right]  ds \le \\ \le C \int_t^{+\infty} \beta(s)
      \left( \int_0^{\mu_{u_n}(s)} \left[(c^+)^*(\tau)u_n^*(\tau)^{p-1}
          +f^*(\tau)  \right] d\tau \right)
      \left(\mu_{u_n}(s)\right)^{-\frac {p-q}{N'}} (-\mu_{u_n}'(s))^{p-q}
    ds,
  \end{multline}
  where last inequality follows from \eqref{eq:est-grad} and
  \eqref{equ}.  We always denote with $C$ a constant
  independent of $n$. As matter of fact, the properties of Lorentz
  spaces give that $f\in L(m,\infty)$ and, by the estimate
  \eqref{reg}, $|u_n|^{p-1}$ are uniformly bounded in
  $L\left(\frac{mN}{N-mp},\infty\right)$. This implies that, for $s\ge
  t$,
  \begin{multline}
    \label{mul:beta2}
    \int_0^{\mu_{u_n}(s)} \left[ (c^+)^*(\tau)u_n^*(\tau)^{p-1} +f^*(\tau)
    \right]d\tau \le
    C \int_0^{\mu_{u_n}(t)} \left[ \tau^{-\frac p N}  \tau^{-\frac{N-mp}{Nm}}
      +\tau^{- \frac 1 m} \right] d\tau=\\=C \mu_{u_n}(t)^{1-\frac 1 m}
    \le
    C t^{-\frac{N(p-1)(m-1)}{N-mp}}.
  \end{multline}
  Hence, applying \eqref{mul:beta2} in \eqref{mul:betaest}, we get
  \[
    \int_{\{|u_n|>t\}} |b_n(x,u_n,Du_n)| dx \le C
    t^{-\frac{N (p-1) (m-1) }{N-mp}} \int_t^{\infty} \beta(s)
    \left(\mu_{u_n}(s)\right)^{-\frac {p-q}{N'}}
    (-\mu_{u_n}'(s))^{p-q} ds.
  \]
Hence, if $q=p$, the thesis follows immediately by
\eqref{Binf}. Otherwise, using the  H\"older inequality, the
hypothesis \eqref{Binf} and again the boundedness of $u_n^{p-1}$ in
$L(\frac{Nm}{N-mp},\infty)$, we get that
  \begin{multline*}
    \int_{\{|u_n|>t\}} |b_n(x,u_n,Du_n)| dx
    \le C t^{-\frac{N(p-1)(m-1)}{N-mp}}
\left(    \int_t^{+\infty} \beta(s)^{\frac{1}{q-(p-1)}} ds \right)^{q-(p-1)}
    \left(\int_0^{\mu_{u_n}(t)} s^{-\frac{1}{N'}} ds \right)^{p-q}
    \\
    \le C t^{-\frac{N(p-1)(m-1)}{N-mp}}  \mu_{u_n}(t)^{\frac {p-q}{N}}
    \le C t^{-\alpha},
  \end{multline*}
with $\alpha = \frac{(p-1)[N(m-1)+m(p-q)]}{N-mp}>0$, and the
proposition is completely proved.
\end{proof}

Now we consider the case $f\in L(m,\sigma)$, with $1<m<(p^*)'$ and
$1< \sigma<+\infty$, and get some estimates for the Lorentz
norm of the gradient of $u_n$. 

\begin{theo}\label{stime2g}
  Suppose that the hypotheses \eqref{ellipt} $\div$ \eqref{ip-c},
  \eqref{Binf} hold. Let $f\in L(m,\sigma)$, with
  $1 < m < (p^*)'$, $p' \le \sigma
  \le +\infty$, and $0\le
  \lambda<\lambda(m)$, with $\lambda(m)$ as in \eqref{def-lambda}.
  Then, the weak solutions $u_n$ of \eqref{eq:gen-app} are such that
  \begin{equation}
    \label{reg2}
    \norm{ H(D u_n)^{p-1}}_{\frac{Nm}{N-m},\sigma}\le C \|f\|_{m,\sigma},
  \end{equation}
  for some constant $C$ independent of $n$.
\end{theo}
\begin{proof}
  We reason similarly to the proof of Theorem \ref{stime1}.
First of all, let $\sigma=+\infty$. Then, recalling \eqref{reg} and
that $(c^+)^*(s) \le \lambda \kappa_N^{\frac p N} s^{-\frac p N}$, we
have
\[
(c^+)^*(s)u_n^*(s)^{p-1} \le C s^{-\frac 1 m}.
\]
Hence, substituting in \eqref{eq:gradlor}, and integrating, we get
that
\[
[H(Du_n)^*(s)]^{p-1} \le C s^{-\frac{N-m}{Nm}},
\]
that gives \eqref{reg2} when $\sigma=+\infty$.

In the case $\sigma<+\infty$, by \eqref{eq:gradlor} we have:
  \begin{multline*}
    \norm{ H(Du_n)^{p-1} }^\sigma_{d,\sigma}=
    \int_0^{|\Omega|} \left(s^{\frac 1 d}
      [H(Du_n)^*(s)]^{p-1} \right)^\sigma \frac{ds}{s}
    \le \\
    C \int_0^{|\Omega|} \left[
      s^{\frac{p'}{d}-\alpha-1} \int_0^s r^{\alpha-\frac{p'}{N'}} \left(
        \int_0^r z(t) dt \right)^{p'} dr \right]^{\frac{\sigma}{p'}}
    \frac{ds}{s}
    +\\ + C \int_0^{|\Omega|}\left[
      s^{\frac {p'}{d}-1} \int_s^{|\Omega|} r^{-\frac{p'}{N'}}  \left(
        \int_0^r z(t) dt \right)^{p'} dr
    \right]^{\frac{\sigma}{p'}} \frac {ds}{s},
  \end{multline*}
  with $z=(c^+)^*(u_n^*)^{p-1}+f_n^*$. Being $\alpha> \frac{p'}{N'}-1$, if
  $N'<d<p'$, using the inequalities \eqref{Hardy1d2} and
  \eqref{Hardy1d1}, we obtain that
  \begin{multline*}
    \norm{ H(Du_n)^{p-1} }^\sigma_{d,\sigma}
    \le C \int_0^{|\Omega|} \left(s^{\frac 1 d + \frac 1 N} z(s)
    \right)^{\sigma} \frac{ds}{s} \le \\
    \le C \int_0^{|\Omega|} \left(s^{\frac 1 d + \frac 1 N - \frac{p}{ N} }
      (u_n^*(s))^{p-1}\right)^{\sigma} \frac{ds}{s} + C
    \int_0^{|\Omega|} \left(s^{\frac
        1 d + \frac 1 N} f_n^*(s) \right)^{\sigma} \frac{ds}{s}.
  \end{multline*}
  Choosing $d$ such that $\frac 1 d + \frac 1 N = \frac 1 m$, we have
  that $d=\frac{Nm}{N-m}$, and $N'< d < p'$, being $1<m< (p^*)'$.

  Being $\lambda<\lambda(m)$, we can use the estimates of Theorem
  \ref{stime1}, obtaining the thesis.
\end{proof}

\section{Proofs of the existence and regularity theorems}

Now we can prove the existence and regularity results for
problem \eqref{eq:gen} stated in Section 3. Using the estimates of the
previous section, we will pass to the limit in the approximating problems
  \begin{equation}\label{eq:gen-appp}
  \left\{
    \begin{array}{ll}
      -\divergenza{a(x,u_n,Du_n)}=b_n(x,u_n,Du_n) + g_n(x,u_n) +
      f_n(x) & \text{in }\Omega, \\
      u_n = 0 & \text{on }\de\Omega,
    \end{array}
  \right.
\end{equation}
where, as in the previous section, $b_n,g_n$ and $f_n$ are the
truncates of $b,g$ and $f$, respectively.

\begin{proof}[Proof of Theorem \ref{theo:deb}]
  As usual, we show that the solutions $u_n\in W_0^{1,p}(\Omega)$ to
  problem \eqref{eq:gen-appp} found in Theorem \ref{stime1} converge
  to a weak solution of \eqref{eq:gen}, i.e. for any $\varphi\in
  W_0^{1,p}(\Omega)\cap L^{\infty}(\Omega)$ it is possible to pass to
  the limit in
  \begin{equation}\label{soldeb}
  \int_\Omega a(x,u_n,Du_n)\cdot D\varphi \, dx =
  \int_\Omega [ b_n(x,u_n,Du_n) + g_n(x,u_n) + f_n] \varphi \, dx.
\end{equation}

  By Theorem \ref{stima0}, being $u_n$ bounded, 
  \begin{multline*}
    e^{-B(\infty)} \int_\Omega H(Du_n)^p dx \le \int_0^{\infty} \left(
      \int_0^{\mu_{u_n}(t)}[(c^+)^*(s) u_n^*(s)^{p-1}+f^*(s)] \,ds \right) dt
    = \\ =
    \int _0^{|\Omega|} (-u_n^*(r))' \left( \int_0^r [(c^+)^*(s)
      u_n^*(s)^{p-1}+f^*(s)]\, ds \right) dr = \\ =
    \lambda \int_{\Omega^\star} \frac{(u_n^\star(x))^p}{H^o(x)^p} dx
    +\int_0^{|\Omega|} f^*(s) u_n^*(s) ds.
  \end{multline*}
  Then, by the Hardy inequality \eqref{Hardy} and the H\"older
  inequality we get
  \[
  e^{-B(\infty)} \int_\Omega H(Du_n)^p dx \le
  \frac{\lambda}{\Lambda_N} \int_{\Omega^\star} H(Du^\star_n)^p dx +
  \|f\|_{(p^*)',p'}\|u_n\|_{p^*,p}.
  \]
  Hence, by the P\'olya-Szeg\"o inequality we get that
  \[
  \left(e^{-B(\infty)}-\frac{\lambda}{\Lambda_N}\right)
  \int_\Omega H(Du_n)^p dx \le   \|f\|_{(p^*)',p'}\|u_n\|_{p^*,p}.
  \]
  Recalling the Remark \ref{remlor}, and being $\lambda<\Lambda_N
  e^{-B(\infty)}$, we get that $u_n$ is uniformly bounded in
  $W_0^{1,p}(\Omega)$ and, up to a subsequence,
  \begin{equation*}
    \begin{array}{ccll}
      u_n & \rightharpoonup &u &\text{weakly in }W^{1,p}_0(\Omega),\\
      g_n(x,u_n) & \rightarrow & g(x,u) & \text{strongly in
      }L^1(\Omega).
    \end{array}
  \end{equation*}
  Moreover, $b_n(x,u_n,Du_n)$ is bounded in $L^1(\Omega)$. Indeed,
  being $\beta$ continuous, using \eqref{eq:uno} of Proposition
  \ref{prop:bound}, we get that
  \begin{multline*}
    \int_\Omega |b_n(x,u_n,Du_n)|dx \le \int_{\{|u_n|\le k \}}
    \beta(u_n) H(Du_n)^q dx + \int_{\{|u_n|>k \}}
    \beta(u_n) H(Du_n)^q dx \\ \le  |\Omega|^{1-\frac qp}\left(
      \int_{\Omega} H(D u_n)^p dx \right)^{\frac qp}  \max_{[0,k]}
    \beta + C k^{-\alpha}.
  \end{multline*}
  Hence, we can apply the compactness result of \cite{bm92}, obtaining
  that $D u_n\rightarrow Du$ a.e. in $\Omega$. Now we prove the strong
  convergence of $b_n(x,u_n,Du_n)$ to $b(x,u,Du)$ in $L^1$. If
  $q=p$, this can be shown by a standard procedure (see for instance
  \cite{lmp} and the references therein). Otherwise, we use the
  Vitali Theorem. The equiintegrability of
  $b_n(x,u_n,Du_n)$ follows observing that, similarly as before,
  \begin{multline*}
     \int_E |T_n(b(x,u_n,Du_n))|dx \le \int_{\{|u_n|\le k \}\cap E}
  \beta(u_n) H(Du_n)^q dx + \int_{\{|u_n|>k \}}
  \beta(u_n) H(Du_n)^q dx \\ \le  |E|^{1-q/p}
  \left( \int_{\Omega} H(D u_n )^p dx \right)^{\frac qp} \max_{[0,k]}
  \beta + C k^{-\alpha} \le C \left( |E|^{1-\frac qp}  \max_{[0,k]}\beta +
    k^{-\alpha} \right).
  \end{multline*}
 Finally, we observe that, recalling \eqref{growth}, $a(x,u_n,Du_n)$
 is bounded in $L^{p'}(\Omega)$, and then it weakly converges to
 $a(x,u,Du)$ in $(L^{p'}(\Omega))^N$. Being $u\in W_0^{1,p}(\Omega)$,
 we can pass to the limit in \eqref{soldeb}, and this
 concludes the proof of part (i).

 Clearly, if $f\in L(m,\sigma)$, $(p^*)'<m<\frac N p$ and
 $\max\{1,\frac{1}{p-1}\}\le \sigma\le
 +\infty$, by Theorem \ref{stime1} the obtained solution $u$ verifies
 the estimate in (ii).
\end{proof}

\begin{proof}[Proof of Theorem \ref{theo:entr}]
  Let $u_n$ be a solution of \eqref{eq:gen-appp}. Then, the estimates
  \eqref{reg} and \eqref{reg2} hold. As matter of fact, for a.e. $t>0$
  we have that
  \[
  - \frac{d}{dt}\int_{\{|u_n|>t\}} H(Du_n)^p dx =
  \frac{d}{dt}\int_{\{|u_n|\le t\}} H(Du_n)^p dx.
  \]
  Hence, applying \eqref{eq:est-grad} to $u_n$, and integrating
  between $0$ and $k$, by \eqref{eq:zero} we have that
  \begin{equation}
    \label{eq:tk}
  \int_\Omega H(D T_k(u_n))^p dx \le C k.
  \end{equation}
  Hence, $T_k(u_n)$ are bounded in $W_0^{1,p}(\Omega)$ and, up to a
  subsequence, $T_k(u_n)\rightharpoonup T_k(u)$ weakly in
  $W_0^{1,p}(\Omega)$. Moreover, $|u|^{p-1}\in L\left(
    \frac{Nm}{N-mp},\sigma \right)$, and $g_n(x,u_n)\rightarrow
  g(x,u)$ strongly in $L^1(\Omega)$.
  Similarly as in the proof of Theorem
  \ref{theo:deb}, by \eqref{eq:uno} we get
  \begin{multline}
    \label{eq:tkg}
     \int_E |b_n(x,u_n,Du_n)|dx \le  |E|^{1-q/p}
  \left( \int_{\Omega} H(D T_k( u_n ) )^p dx \right)^{\frac qp} \max_{[0,k]}
  \beta + C k^{-\alpha} \le \\ \le C \left( |E|^{1-\frac qp} k^{\frac
      q p} \cdot \max_{[0,k]}\beta + k^{-\alpha} \right),
\end{multline}
where last inequality follows from \eqref{eq:tk}. Then, \eqref{eq:tkg}
gives that $b_n(x,u_n,Du_n)$ is bounded in $L^1(\Omega)$, and by
the compactness result contained in \cite{bg92} (see also
\cite{dv95}), $Du_n\rightarrow Du$ a.e. in $\Omega$ (up to a
subsequence). If $q<p$, the Vitali Theorem assures the strong
convergence of $b_n(x,u_n,Du_n)$ to $b(x,u,Du)$ in $L^1(\Omega)$,
and the strong convergence of $T_k(u_n)$ to $T_k(u)$ in
$W_0^{1,p}(\Omega)$. Otherwise, for $q=p$ this can be shown in a
standard way using a suitable exponential test function (see for
instance \cite{lmp}, and the reference therein). Hence, we can pass to
the limit in the right-hand side of
\begin{equation}
  \label{defsol-entr2}
  \int_\Omega a(x,u_n,Du_n)\cdot D T_k(u_n-\varphi) \, dx =
  \int_\Omega [ b_n(x,u_n,Du_n) + g_n(x,u_n) + f_n ] T_k (
  u_n - \varphi ) \, dx,
\end{equation}
for all $\varphi\in W_0^{1,p}(\Omega)\cap L^{\infty}(\Omega)$.

As regards the left-hand side of \eqref{defsol-entr2}, a simple
argument based on the Fatou Lemma (see \cite{boc96}) allows to show
that
\[
\int_{\Omega}a(x,u,Du) \cdot D T_k(u-\varphi) dx \le
\liminf_{n\rightarrow +\infty} \int_\Omega a(x,u_n,Du_n)\cdot D
T_k(u_n-\varphi) \, dx.
\]
Then $u$ verifies \eqref{defsol-entr}, and this concludes the proof
of the existence result.

Finally, by Theorem \ref{stime2g} the obtained solution $u$ verifies
 \eqref{estg}.
\end{proof}
\begin{rem}
  \label{remex}
  In order to get an existence and regularity result for $f\in L^m$,
  $1<m<(p^*)'$, we can repeat line by line the above proof using the
  estimates contained in Remark \ref{remstime}.
\end{rem}

\end{document}